\documentclass[11pt]{amsart}
\usepackage{hyperref}
\usepackage[capitalize, nameinlink]{cleveref}
\crefname{theorem}{Theorem}{Theorems}
\crefname{proposition}{Proposition}{Propositions}
\crefname{observation}{Observation}{Observations}
\crefname{lemma}{Lemma}{Lemmas}
\crefname{claim}{Claim}{Claims}
\crefname{problem}{Problem}{Problems}
\crefname{conjecture}{Conjecture}{Conjectures}
\crefname{question}{Question}{Questions}
\crefname{example}{Example}{Examples}
\crefname{fact}{Fact}{Facts}

\usepackage{amssymb}
\usepackage[bbgreekl]{mathbbol}
\usepackage{graphicx}
\usepackage{ifthen}
\usepackage{pict2e}
\usepackage{xargs}
\usepackage{xspace}
\usepackage{xcolor}
\usepackage{pgf,tikz}
\usepackage{pgfplots}
\usepackage{environ}
\usepackage{caption}
\usetikzlibrary{arrows}
\usepackage{enumitem}
\usepackage{xifthen}
\usepackage{comment}
\newcounter{dummy}
\makeatletter
\newcommand\myitem[1][]{\item[#1]\refstepcounter{dummy}\def\@currentlabel{#1}}
\makeatother

\makeatletter
\newsavebox{\measure@tikzpicture}
\NewEnviron{scaletikzpicturetowidth}[1]{%
	\def\tikz@width{#1}%
	\begin{lrbox}{\measure@tikzpicture}%
		\BODY
	\end{lrbox}%
	\pgfmathparse{#1/\wd\measure@tikzpicture}%
	\BODY
}
\makeatother

\DeclareSymbolFontAlphabet{\mathbb}{AMSb}

\newcommand{\thistheoremname}{}
\newtheorem*{genericthm*}{\thistheoremname}
\newenvironment{namedthm*}[1]
{\renewcommand{\thistheoremname}{#1}%
	\begin{genericthm*}}
	{\end{genericthm*}}

\newcommand{\Bairespace}[1][]{
	\ifthenelse{\equal{#1}{}}{\functions{\N}{\N}}{\functions{#1}{\N}}
}
\newcommand{\bbL}{\mathbb{L}}
\newcommand{\bbX}{\mathbb{X}}

\newcommand{\Cantorspace}[1][]{
	\ifthenelse{\equal{#1}{}}{\functions{\N}{2}}{\functions{#1}{2}}
}

\newcommandx{\concatenation}[2][1 = undefined, 2 = undefined]{
	\ifthenelse{\equal{#1}{undefined}}{{}\smallfrown}{
		\ifthenelse{\equal{#2}{undefined}}{\bigoplus #1}{\bigoplus_{#1} #2}
	}
}

\newcommandx{\functions}[3][3 =]{
	\ifthenelse{\equal{#3}{}}{#2^{#1}}{#2_{#3}^{#1}}
}

\newcommand{\Gzero}[1][]{
	\ifthenelse{\equal{#1}{}}
	{\mathbb{G}_0}
	{\mathbb{G}_{0,n}}
}
\newcommandx{\Hzero}[2][2 = undefined]{
	\ifthenelse{\equal{#2}{undefined}}
	{\mathbb{H}_{#1}}
	{\mathbb{H}_{#1, #2}}
}
\newcommandx{\intersection}[2][1 =, 2 =]{
	\ifthenelse{\equal{#1}{}}{\cap}{
		\ifthenelse{\equal{#2}{}}{\bigcap #1}{{\bigcap_{#1} #2}}
	}
}
\newcommand{\Lzero}[1][]{\ifthenelse{\equal{#1}{}}{\bbL_0}{L_{0, #1}}}
\newcommand{\Lzerospace}[1][]{\ifthenelse{\equal{#1}{}}{\bbX_0}{X_{0, #1}}}

\newcommand{\modulo}[1]{\ (\text{mod } 2)}
\newcommand{\N}{\mathbb{N}}

\newcommandx{\product}[2][1 =, 2 =]{
	\ifthenelse{\equal{#1}{}}{\times}{
		\ifthenelse{\equal{#2}{}}{\prod #1}{{\prod_{#1} #2}}
	}
}

\newcommandx{\sequence}[2][2 = undefined]{
	\ifthenelse{\equal{#2}{undefined}}{(#1)}{
		(#1)_{#2}
	}
}

\newcommandx{\set}[2][2 = undefined]{
	\ifthenelse{\equal{#2}{undefined}}{\{ #1 \}}{
		\{ #1 \suchthat #2 \}
	}
}
\newcommandx{\sets}[3][3 =]{
	\ifthenelse{\equal{#3}{}}{[#2]^{#1}}{[#2]^{#1}_{#3}}
}

\newcommand{\suchthat}{\mid}

\renewcommand{\restriction}[2]{#1 \upharpoonright #2}

\newcommandx{\union}[2][1 =, 2 =]{
	\ifthenelse{\equal{#1}{}}{\cup}{
		\ifthenelse{\equal{#2}{}}{\bigcup #1}{{\bigcup_{#1} #2}}
	}
}

\newtheorem{theorem}{Theorem}[section]
\newtheorem{lemma}[theorem]{Lemma}

\newtheorem{claim}[theorem]{Claim}

\newtheorem{corollary}[theorem]{Corollary}
\newtheorem{proposition}[theorem]{Proposition}

\newtheorem{problem}[theorem]{Problem}

\theoremstyle{definition}

\newtheorem{definition}[theorem]{Definition}

\newtheorem{remark}[theorem]{Remark}

\numberwithin{equation}{section}

\DeclareMathOperator{\maxdeg}{maxdeg}

\newcommand{\bd}{\begin{definition}}
	\newcommand{\ed}{\end{definition}}

\DeclareMathOperator{\graph}{graph}

\DeclareMathOperator{\dist}{dist}
\DeclareMathOperator{\didistance}{didist}
\newcommand{\mc}{\mathcal}

\newcommand{\distance}[3]{\ifthenelse{\isempty{#3}}{\dist(#1,#2)}{\dist^{#3}(#1,#2)}}
\newcommand{\didist}[3]{\ifthenelse{\isempty{#3}}{\didistance(#1,#2)}{\didistance^{#3}(#1,#2)}}
\newcommand{\digraph}[3]{\ifthenelse{\equal{#1}{b}}{\mathbb{#2}_{#3}}
	{{#2}_{#3}}}
\newcommand{\linegraph}[3]{\ifthenelse{\equal{#1}{b}}{\mathbb{#2}_{#3}}
	{#2_{#3}}}

\newcommand{\underlyingspace}[3]{\ifthenelse{\equal{#1}{b}}{\mathbb{#2}_{#3}}
	{#2_{#3}}}
\newcommand{\distanceset}[2]{\ifthenelse{\isempty{#2}}{D(#1)}{D^{#2}(#1)}}

\newcommand{\rstr}{\restriction}
\newcommand{\vfi}{\varphi}
\newcommand{\sse}{\subseteq}
\newcommand{\est}{\emptyset}

\newcommand{\fA}{\mathcal{A}}

\newcommand{\fR}{\mathcal{R}}
\newcommand{\fS}{\mathcal{S}}

\begin{document}

	\thanks{}
	
	\keywords{}
	
	\subjclass[2020]{Primary 03E15, Secondary 68Q17}
	
	\title[]{Hyperfiniteness on Topological Ramsey Spaces}
	
	\author{Bal\'azs Bursics}
	
	\address{E\"otv\"os Lor\'and University, Institute of Mathematics, P\'azm\'any P\'eter stny. 1/C, 1117 Budapest, Hungary}
    \email{bursicsb@gmail.com}

	\author{Zolt\'an Vidny\'anszky}
    
	\address{E\"otv\"os Lor\'and University, Institute of Mathematics, P\'azm\'any P\'eter stny. 1/C, 1117 Budapest, Hungary}
    \email{zoltan.vidnyanszky@ttk.elte.hu}

	
	
	\maketitle
	
	\begin{abstract}
		
		We investigate the behavior of countable Borel equivalence relations (CBERs) on topological Ramsey spaces. First, we give a simple proof of the fact that every CBER on $[\N]^\N$ is hyperfinite on some set of the form $[A]^\N$. Using the idea behind the proof, we show the analogous result for every topological Ramsey space.

	\end{abstract}

	\maketitle
	\section{Introduction}
	A \emph{countable Borel equivalence relation (CBER)} on a Polish space $X$ is an equivalence relation with all classes countable that is Borel as a subset of $X^2$. CBERs have a rich and well investigated structure theory (see, e.g., \cite{kechris2024theory}). The natural way to compare two such equivalence relations is the notion of \emph{reduction}: if $E$ and $F$ are CBERs on spaces $X$ and $Y$, $E$ is said to \emph{Borel reduce} to $F$ if there exists a Borel map $\phi:X \to Y$ with
	\[\forall x,x' \in X \ xEx' \iff \phi(x)F\phi(x').\]
	
	The simplest non-trivial class of CBERs is the collection of \emph{hyperfinite} ones, that is, the collection of CBERs which admit a Borel reduction to $\mathbb{E}_0$, the eventual equality equivalence relation on $2^\N$. 
	
	It is a classical fact that there are CBERs that do not reduce to $\mathbb{E}_0$, however, so far, essentially only measure theoretic proofs\footnote{Results of Thomas \cite{thomas2009martin} provide a non-measure theoretic argument, but rely on Martin's conjecture.} are known to establish this fact. This motivates the investigation of CBERs with respect to different notions of largeness. For example, classical results of Hjorth-Kechris \cite{hjorth-kechris1996} and Sullivan-Weiss-Wright \cite{sullivan1986generic} establish hyperfiniteness on comeager sets, and more recently Panagiotopoulos-Wang \cite{panagiotopoulos2022every} and Marks-Rossegger-Slaman \cite{marks2024hausdorff} have shown smoothness (that is, reducibility to equality on $2^\N$) of CBERs restricted to positive sets with respect to the Carlson-Simpson ideal and Hausdorff dimension, respectively. The monograph of Kanovei-Sabok-Zapletal \cite{kanovei2013canonical} contains a large number of deep canonization-type results, mostly in the context of (not necessarily countable) Borel equivalence relations.  
	
	In this paper, we consider CBERs on topological Ramsey spaces in the sense of Todor\v{c}evi\'c, see \cite{ramseyspaces}. The prime example of these spaces is the \emph{Ellentuck space} $[\N]^\N$, that is, the collection of infinite subsets of the natural numbers, endowed with the topology inherited from $2^\N$. The following theorem has been shown by Soare \cite{soare1969sets} in the recursion theoretic context and Mathias (see Kanovei-Sabok-Zapletal \cite{kanovei2013canonical}). 
	
	\begin{theorem}
		\label{t:ellentuck}
		Let $E$ be a CBER on $[\N]^\N$. There exists a set $A \in [\N]^\N$ such that $\restriction{E}{[A]^\N}$ is hyperfinite. 
	\end{theorem}
	
	The proof of a slightly more general fact than the above one in Kanovei-Sabok-Zapletal \cite{kanovei2013canonical} uses forcing, and a non-straightforward trick with ``even-odd" encoding. We give a few line elementary argument, the idea behind which generalizes to arbitrary topological Ramsey spaces, yielding the following: 
	
	\begin{theorem} \label{t:ramsey hyperfin}
		Let $\mathcal{R}$ be a topological Ramsey space and $E$ be a CBER on $\mathcal{R}$. Then there is a set $A \in \mc{R}$ such that $\restriction{E}{[A]}$ is hyperfinite. 
	\end{theorem}
	
	Along the way we also isolate a condition that could be interesting on its own. A version of this notion has been considered by Marks and Unger \cite{marksungerbaire}, to establish a so-called toast structure on co-meager sets, with respect to bounded degree Borel graphs (see also \cite{gao2022forcing,conley2016bound,grebik_rozhon2021toasts_and_tails,brandt_chang_grebik_grunau_rozhon_vidnyaszky2021LCLs_on_trees_descriptive}).
	
	Recall the following definition. 
	\begin{definition}
		Let $G$ be a graph, $B \subseteq V(G)$. The set $B$ is called \emph{$k$-separated} if for any $x \neq x' \in B$ we have $\dist_{G}(x,x')>k$.
	\end{definition}
	
	The following observation is our main tool towards Theorem \ref{t:ramsey hyperfin}. Roughly speaking, it says that if the space can be covered by very sparse sets, w.r.t. a bounded degree decomposition of $E$ so that every point is covered infinitely often, then $E$ is hyperfinite. 
	
	\begin{theorem}
		\label{t:separation2}
		Let $E$ be a CBER on the space $X$, let $(G_n)_{n \in \N}$ be an increasing sequence of bounded degree Borel graphs\footnote{Throughout the paper, graphs are not necessarily assumed to be irreflexive.} such that $\bigcup_n G_n=E$, and let $f:\N\to\N$ be such that $\forall n\in\N\; f(n+1)\ge2\cdot(f(n)+1)$. Moreover, assume that $B_n \subseteq X$ are Borel sets so that every $B_n$ is $f(n)$-separated in $G_n$. 
		
		Then $\restriction{E}{B}$ is hyperfinite, where $B=\{x:\exists^\infty n \ (x \in B_n)\}$.
	\end{theorem}
	
	It was already suggested by Marks and Unger that such a construction yields a proof of the theorem of Hjorth-Kechris \cite{hjorth-kechris1996} and Sullivan-Weiss-Wright \cite{sullivan1986generic} mentioned above that every CBER on a Polish space is hyperfinite on a co-meager set. In fact, unlike Segal's proof (see \cite{kechris2004topics}), this proof does not rely on the Kuratowski-Ulam theorem. We include this argument for the sake of completeness.
	\subsection*{Roadmap} In Section \ref{s:prel} we collect all the facts we need about topological Ramsey spaces. Section \ref{s:ellentuck} contains a short argument establishing the canonization result on the space $[\N]^\N$. Then, in Section \ref{s:main} we show our theorem about sufficiently separated covers and then apply it to obtain hyperfiniteness on topological Ramsey spaces and on comeager sets. Finally, Section \ref{s:problems} contains some open problems. Sections \ref{s:ellentuck}, \ref{s:suffsep} and \ref{s:generic} do not use the definition of a topological Ramsey space.

	\subsection*{Acknowledgments} 
	The authors were supported by Hungarian Academy of Sciences Momentum Grant no. 2022-58 and National Research, Development and Innovation Office (NKFIH) grant no.~146922.
	
	\section{Preliminaries}
	\label{s:prel}

	\subsection{Topological Ramsey spaces}

	A \emph{topological Ramsey space} in the sense of Todor\v{c}evi\'c \cite{ramseyspaces} is a triple $(\mathcal{R},\le,r)$ satisfying a certain set of axioms, where $\mathcal{R}$ is a nonempty set, $\le$ a quasi-ordering on $\mathcal{R},$ and $r:\mathcal{R}\times\N\to\mathcal{AR}$. The range of $r$ can be thought of as the collection of finite approximations to elements of $\mathcal{R}$. We will use capitals $A,B,\ldots$ for the elements of $\mathcal{R}$, and $a,b,\ldots$ for their approximations. We denote $r_n(.)=r(.,n)$.
	
	The relation $\le$ is encoded on the level of approximations by a quasi-ordering $\le_{\text{fin}}$ on $\mathcal{AR}$.
	
	For $a\in\mathcal{AR}$ and $B\in\mathcal{R}$ define 
	$$[a,B]=\{A\in\mathcal{R}:A\le B \wedge (\exists n)\; r_n(A)=a\}.$$
	
	We also use the abbreviations $[B]=\{A\in\mathcal{R}:A\le B \}$ and $[n,B]=[r_n(B),B].$
	
	There are multiple natural topologies on $\fR$: the \emph{metric topology} of $\fR$ inherited from the Baire space $\fA\fR^\N,$ and the so-called \emph{Ellentuck topology} generated by $\{[a,A]:\; a\in\fA\fR,\; A\in\fR\}.$ We work with the metric topology, unless stated otherwise.

	The next definition gives a standard way of constructing elements of $\fR$ with some desirable property.
	
	\begin{definition}
		A \emph{fusion sequence} is a sequence $(A_n)_{n\in \N}$ in $\fR$ such that $A_{n+1}\in [n,A_n]$ for every $n\in\N.$
		
		The \emph{limit} of a fusion sequence $(A_n)_{n\in\N}$, denoted $\lim A_n$, is the unique $A\in\fR$ such that $A\in[n,A_n]$ for every $n\in\N.$ 
	\end{definition}
	
	Note that for the existence of the limit of a fusion sequence we assume that $\fR$ is closed in $\fA\fR^\N$ with respect to the metric topology. This is not listed among the axioms of topological Ramsey spaces in \cite{ramseyspaces}, however, this property is assumed for the main results of the theory. We make use of the following properties of fusion sequences.

	\begin{proposition}\label{fusion}
		Let $(A_n)$ be a fusion sequence and $C\in[\lim A_n].$ Then 
        \begin{enumerate}
            \item $\lim A_n\le A_k$ for every $k$,
            \item there are infinitely many $n$ with $C\in[s,A_n]$ for some $s\le_{fin}r_n(A_n).$
        \end{enumerate}
	\end{proposition}
	
	Apart from the technique of fusion sequences, we use the Ramsey theoretic statement below.
	
	\begin{theorem}\label{ramsey property}
		Let $\fR$ be a topological Ramsey space, $c:\fR\to k$ a finite Borel coloring, $A\in\fR,$ and $s\le_{fin}r_n(A)$. Then there exists $B\in[n,A]$ such that $\rstr{c}{[s,B]}$ is monochromatic.
	\end{theorem}

	\subsection{The Ellentuck space} The space $[\N]^\N$ has a topological Ramsey structure by defining $\le$ and $r$ as follows:

    \begin{itemize}
        \item $A\le B\iff A\sse B$
        \item $r_n(A)=\{a_k: k< n\}$, where $A=\{a_k:k\in\N\}$ with $a_k<a_{k+1}$ for every $k\in\N$
    \end{itemize}
    
    The notation $[a,B]$ defined above for $A\in[\N]^\N$ and $a\in[\N]^n$ gives $[a,B]=\{A\sse B: r_n(A)=a\},$ but instead we stick to the classical notation \[[a,B]=\{A\in[\N]^\N: a\subset A\sse a\cup B\}\] where we require $\max a<\min B.$
	
	A fusion sequence in $[\N]^\N$ is a sequence $(A_n)_{n\in\N}$ so that $A_{n+1}\sse A_n$ for every $n\in\N$ and the set of $n$ smallest elements is the same for every $A_m$ with $m\ge n$, and $\lim A_n =\cap A_n$. 
	
	Here Theorem \ref{ramsey property} corresponds to the well-known theorem of Galvin-Prikry.
	It is a classical result of Silver and also a consequence of the theory of topological Ramsey spaces that the conclusion of the Galvin-Prikry theorem remains true even when we consider colorings that are \emph{Souslin measurable}, i.e., that are measurable w.r.t. the $\sigma$-algebra generated by analytic sets. A further consequence of this result, using a standard fusion argument, is the following.
	
	\begin{theorem}
		\label{t:function} Let $f:[t,A] \to [\N]^\N$ be a Souslin measurable function. Then there is some $B \in [A]^\N$ such that $\restriction{f}{[t,B]}$ is continuous.
	\end{theorem}

\subsection{General Ramsey spaces} \label{general ramsey}
	
	In the monograph \cite{ramseyspaces} a more general version of topological Ramsey spaces, namely, \emph{Ramsey spaces} are also discussed. These allow a stronger form of Theorem \ref{ramsey property}, where the colored objects differ from the base objects of canonization.
	
    A Ramsey space is a tuple $(\mathcal{R},\mathcal{S}, \le,\le^\circ,r,s)$ obeying a certain set of axioms, where $\mathcal{R},\mathcal{S}$ are nonempty sets, $r:\mathcal{R}\times\N\to\mathcal{AR}$ and $s:\mathcal{S}\times\N\to\mathcal{AS}$ are approximation functions, $\le$ is a quasi-ordering on $\mathcal{S},$ and $\le^\circ$ is a subset of $\mathcal{R}\times\mathcal{S}.$ Here $\mathcal{R}$ is the set of objects whose colorings can be canonized, and $\le^\circ$ determines the connection between $\mathcal{R}$ and $\mathcal{S}.$

    For $a\in\mathcal{R}$ and $X\in \mathcal{S}$ let
    $$[a,X]=\{A\in\mathcal{R}: A\le^\circ X\wedge(\exists n)r_n(A)=a\},$$
    specifically, 
    $$[\est,X]=\{A\in\mathcal{R}:A\le^\circ X\}.$$

    Here, the simplest Ramsey-type canonization theorem is the following.

    \begin{theorem}
		Let $(\fR,\fS,\le,\le^\circ,r,s)$ be a Ramsey space, and $c:\fR\to k$ a finite Borel coloring. Then there exists $X\in\fS$ such that $\rstr{c}{[\est,X]}$ is constant.
	\end{theorem}

	\section{The case of $[\N]^\N$}
	\label{s:ellentuck}
	In this section, we give a simple proof of the canonization theorem on the Ellentuck space and reprove the mentioned results of Soare and  Mathias. This relies on the following easy lemma. 
	
	\begin{lemma}
		\label{l:suslin} Let $G$ be an irreflexive bounded degree Borel graph on $[\N]^\N$, and suppose that $\max t<\min A$ for $t \in [\N]^n$ and $A\in [\N]^\N$. Then there is $A' \in [t,A]$ such that $[t,A'\setminus t]$ is $G$-independent. 
	\end{lemma}
	\begin{proof}
		By a classical result of Kechris-Solecki-Todor\v{c}evi\'c \cite{KST}, every irreflexive bounded degree Borel graph admits a finite Borel vertex coloring. Let $c:[\N]^\N \to k$ be such a coloring of $G$. By the Galvin-Prikry theorem (or, Theorem \ref{ramsey property}), there is some $A' \in [t, A]$ and $j$ with $[t,A'\setminus t] \subseteq c^{-1}(j)$. Then $[t,A'\setminus t]$ cannot contain a $G$-edge.
	\end{proof}
	
	A slightly more sophisticated version will be sufficient for our purposes. 
	
	\begin{lemma}
		\label{l:indep2} Let $G$ be a bounded degree Borel graph on $[\N]^\N$, and suppose that $\max t<\min A$ for $t \in [\N]^n$ and $A\in [\N]^\N$. Then there is $A' \in [t,A]$ such that $\restriction{G}{[A']^\N} \subseteq \mathbb{E}_0.$
	\end{lemma}
	\begin{proof}
		We define an irreflexive graph $G^*$ on $[t,A]$ by 
		\begin{align*}
		(B,C)\in G^*&\iff \big(B\ne C \land \exists r,s \subseteq t: (B \setminus r,C \setminus s)\in G\big).
		\end{align*}
		
		Clearly, $\maxdeg G^* \leq 2^{n} \maxdeg G$, so by Lemma \ref{l:suslin} there is some $A' \in[t,A]$ such that $[t,A'\setminus t]$ is $G^*$-independent. But this means that $B,C \in [A']^\N$ cannot be $G$-related, unless $t \cup B=t \cup C$, yielding $B \mathbb{E}_0 C$.
	\end{proof}
	Now let us finish the proof of Theorem \ref{t:ellentuck}, that is:
	\begin{namedthm*}{Theorem \ref{t:ellentuck}}
			Let $E$ be a CBER on $[\N]^\N$. There exists a set $A \in [\N]^\N$ such that $\restriction{E}{[A]^\N}$ is hyperfinite. 
	\end{namedthm*}
	\begin{proof}
		Suppose that $E$ is a countable Borel equivalence relation on $[\N]^\N$. As all classes of $E$ are countable, there are Borel involutions $\vfi_n$ on $[\N]^\N$ such that $E=\bigcup_{n\in\N}\graph(\vfi_n)$. It is sufficient to construct a fusion sequence $(A_n)_{n\in\N}$ such that $\rstr{\graph(\vfi_n)}{[A_n]^\N} \subseteq\mathbb{E}_0,$ as $A=\lim A_n$ will be as required. If $A_n=\{a_k:k\in\N\}$ is given with $a_k<a_{k+1}$ for every $k\in\N$, letting $t=\{a_k:k<n\}$ and applying Lemma \ref{l:indep2} to $G=\graph(\varphi_{n+1})$ and $[t,A_n\setminus t]$ yields the desired $A_{n+1}$.
	\end{proof}
	
	The original motivation behind Soare's work was recursion theoretic. His aim was to show the existence of a set $A \in [\N]^\N$ such that it does not contain a subset of higher Turing-degree. Lemma \ref{l:indep2} yields this statement as well.
	
	\begin{corollary}
		There exists a set $A\in[\N]^\N$ such that if $A \leq_T B$ and $B \in [A]^\N$ then $A \mathbb{E}_0 B$.
	\end{corollary}
	
	\begin{proof}
		Towards contradiction, assume that such a set does not exist. It implies that for each $A \in [\N]^\N$ there exists some $B \in [A]^\N$ such that $|A \setminus B|=\infty$ and $B \geq_T A$. In particular, there is some $m$ with $A=\varphi_m(B)$, where $\varphi_m$ is the $m$th Turing functional. Then, there is a set of the form $[t,A]$ such that the same $m$ works for all $A' \in [t,A]$. Since $\varphi_m$ is Borel, by the Jankov-von Neumann uniformization theorem there is a Souslin measurable map $\psi: [t,A] \to [\N]^\N$ such that $A'=\varphi_m(\psi(A')),$ for all $A' \in [t,A].$ Moreover, by Theorem \ref{t:function} (by passing to a further set of the form $[t,B]$) we may assume that $\psi$ is continuous on $[t,A]$. Observe that as $\psi$ is a partial inverse, it is one-to-one. Hence, by the Luzin-Novikov theorem, the graph $G$ on $[t\cup A]^\N$ defined by \[(B,C) \in G \iff B,C \in [t,A] \cup \psi([t,A]) \land (B=\psi(C) \lor C=\psi(B))\] is Borel and has degrees bounded by $2$. By Lemma \ref{l:indep2} there is some $A' \in [t,A]$ so that $\restriction{G}{[A']^\N}$ is contained in $\mathbb{E}_0$. This contradicts that $A' \setminus \psi(A')$ is infinite. 
	\end{proof}

	\section{General topological Ramsey spaces}
	
	\label{s:main}
	
	In this section we prove the hyperfiniteness result for all topological Ramsey spaces. 
	\subsection{Sufficiently separated covers}
	\label{s:suffsep}
	
	Let us first prove the following:
	\begin{namedthm*}{Theorem \ref{t:separation2}}
		
		Let $E$ be a CBER on the space $X$, let $(G_n)_{n \in \N}$ be an increasing sequence of bounded degree Borel graphs such that $\bigcup_n G_n=E$, and let $f:\N\to\N$ be such that $\forall n\in\N\; f(n+1)\ge2\cdot(f(n)+1)$. Moreover, assume that $B_n \subseteq X$ are Borel sets so that every $B_n$ is $f(n)$-separated in $G_n$. 
		
		Then $\restriction{E}{B}$ is hyperfinite, where $B=\{x:\exists^\infty n \ (x \in B_n)\}$.
	\end{namedthm*}
	
	\begin{proof}
		
		We define a sequence $H_0 \subseteq H_1 \subseteq \dots \subseteq H_n \subseteq \dots \subseteq E$ inductively. Let $H_0=\emptyset$. 
		
		Given $H_n$, let \[H_{n+1}=H_{n} \cup \{(x,y) \in \restriction{G_{n+1}}{B}:\; x\in B_{n+1} \text{ or } y\in B_{n+1}\}.\]
		
		First we check that $\rstr{\bigcup_{n \in \N} H_n}{B}=\restriction{E}{B}$. Indeed, let $(x,y) \in \rstr{G_n}{B}$. Then there is some $k \geq n$ such that $x \in B_k$. But then $(x,y) \in H_k$. 
		
		\begin{claim}
			Every connected component of $H_n$ has diameter $\le f(n)$. 
			
		\end{claim}
		\begin{proof}
			We prove this by induction on $n$. The case $n=0$ is clear, so assume that we have shown the statement for $n$. We claim that if $H$ is a connected component of $H_{n+1}$ then it contains at most one point of $B_{n+1}$. Indeed, otherwise let $x \neq x' $ from $B_{n+1}$ in $V(H)$ be such that for $p=(x_0,\dots,x_k)$, the injective path with minimal length from $x_0=x$ to $x_k=x'$ in $H$ we have $x_1,\ldots,x_{k-1}\notin B_{n+1}$. Then $x_1,\ldots,x_{k-1}$ are contained in a single connected component $H'$ of $H_n$, and by the inductive hypothesis $\dist_{H_n}(x_1,x_{k-1})\le f(n).$ But this implies $\dist_{G_{n+1}}(x,x')\le f(n)+2\le f(n+1),$ a contradiction.
			
			Thus, $H$ contains at most one point, say $x$, from $B_{n+1}$. Then \[V(H)=\bigcup \{[y]_{H_{n}}:(x,y) \in G_{n+1}\}.\]
			yielding \[\text{diam}(H) \le 2\cdot (f(n)+1) \le f(n+1).\]
		\end{proof}
		
		As $H_n\sse G_n$ which has bounded degrees, this, in turn also ends the proof of Theorem \ref{t:separation2}.
	\end{proof}
	
	\subsection{On topological Ramsey spaces}
	\label{s:topram}
	Now we show that we can construct sufficiently separated covers on positive sets on every topological Ramsey space.

	\begin{lemma}
		\label{l:graph canonization}
		Let $\fR$ be a topological Ramsey space, let $G$ be a bounded degree Borel graph on $\fR$, let $k\in \N$, and let $a\in \fA\fR$ and $A\in\fR$ be such that $a\le_{\text{fin}}r_n(A)$. Then there exists $A'\in [n,A]$ such that $[a,A']$ is $k$-separated.
	\end{lemma}

	\begin{proof}
		By \cite{KST}, every irreflexive bounded degree Borel graph admits a finite Borel vertex coloring. Let $c:\fR \to m$ be such a coloring of \[G^k=\{(x,y):x\neq y \wedge \dist(x,y)\le k \}.\] 
        By Theorem \ref{ramsey property}. there is some $A'\in[n,A]$ such that $\rstr{c}{[a,A']}$ is constant, thus, $[a,A']$ is $k$-separated. 
	\end{proof}

	\begin{lemma}\label{ramsey separation}
		Let $\fR$ be a topological Ramsey-space, $G_n$ a sequence of bounded degree Borel graphs on $\fR$, and $f:\N\to\N$. Then there exist $A\in\fR$ and $B_n\sse\fR$ Borel sets such that every $B_n$ is $f(n)$-separated in $G_n$, and 
		$$[A]\sse\{C:\exists^\infty n \ (C \in B_n)\}.$$
	\end{lemma}
	
	\begin{proof}
		We construct a fusion sequence $(A_n)_{n\in\N}$ inductively. Let 
		\[k_n=\sum_{i<n}|\{s\in\fA\fR: \; s\le_{\text{fin}}r_i(A_i)\} |.\]
		Let $B_0=\est.$ Throughout the construction, we define $B_{k_n+1}, \ldots, B_{k_{n+1}}$ after fixing $A_n$ (and  $k_{n+1}$).
		Suppose that $A_n$ and $B_0,\ldots, B_{k_n}$ are given. We enumerate $\{s\in\fA\fR: \; s\le_{\text{fin}}r_n(A_n)\}$ as $\{s_j: 1\le j\le k_{n+1}-k_n\}$, and define $(A_j')_{j\le k_{n+1}-k_n}$ recursively as follows: let $A_0'=A_n$, and if $A_j'$ is already given for some $j<k_{n+1}-k_n$, choose $A_{j+1}'\in [n,A_j']$ so that $[s_{j+1},A_{j+1}']$ is $f(k_n+j+1)$-separated in $G_{k_n+j+1}$ (such an element of $\fR$ exists by Lemma \ref{l:graph canonization}), and let $B_{k_n+j+1}=[s_{j+1},A_{j+1}']$. Afterwards, let $A_{n+1}=A_{k_{n+1}-k_n}'\in [n,A_n].$
		
		Finally, let $A=\lim A_n$ and suppose that $C'\in [A].$ By Proposition \ref{fusion}. there are infinitely many $n$ with $C'\in[s,A_n]$ for some $s\le_{fin}r_n(A_n)$. For such an $n$ and $s$, we have that $[s,A_n]\sse B_k$ for some $k_{n-1}<k\le k_n,$ which yields $C' \in\{C:\exists^\infty n \ (C \in B_n)\}$.
	\end{proof}

		Now we are ready to show our main result. 
	\begin{namedthm*}{Theorem \ref{t:ramsey hyperfin}}
		Let $\mathcal{R}$ be a topological Ramsey space and $E$ be a CBER on $\mathcal{R}$. Then there is a set $A \in \mc{R}$ such that $\restriction{E}{[A]}$ is hyperfinite. 
	\end{namedthm*}
	
	\begin{proof}
		As all classes of $E$ are countable, there are Borel involutions $\vfi_n$ on $\fR$ such that $E=\bigcup_{n\in\N}\graph(\vfi_n)$. Let $G_n=\bigcup_{i\in n}\graph(\vfi_i)$, and $f:\N\to \N$ be as in Theorem \ref{t:separation2}. By Lemma \ref{ramsey separation}. there exist $A\in\fR$ and $B_n\sse\fR$ Borel sets such that every $B_n$ is $f(n)$-separated in $G_n$, and 
		$$[A]\sse\{C:\exists^\infty n \ (C \in B_n)\},$$
		thus, $\rstr{E}{[A]}$ is hyperfinite according to Theorem \ref{t:separation2}.
	\end{proof}

	\begin{remark}
	    It is straightforward to modify our proof to the more general setting of Ramsey spaces described in Subsection \ref{general ramsey} to yield the following: for any Ramsey space $(\mathcal{R}, \mathcal{S}, \le,\le^\circ, r, s)$ and $E$ CBER on $\mathcal{R}$ there exists $X\in\mathcal{S}$ such that $\rstr{E}{[\est,X]}$ is hyperfinite.
	\end{remark}
	
	\subsection{Generic hyperfiniteness}
	\label{s:generic}
	
	As mentioned above, the method of sufficiently separated covers also gives a new proof of the classical results of Hjorth-Kechris \cite{hjorth-kechris1996} and Sullivan-Weiss-Wright \cite{sullivan1986generic}.
	
	\begin{theorem}
		Let $E$ be a CBER on the space $X$. Then there is a comeager invariant Borel set $C$ such that $\rstr{E}{C}$ is hyperfinite.
	\end{theorem}
	
	\begin{proof}
		As all classes of $E$ are countable, there are Borel involutions $\vfi_n$ on $X$ such that $E=\bigcup_{n\in\N}\graph(\vfi_n)$. Let $G_n=\bigcup_{i\in n}\graph(\vfi_i)$, and $f:\N\to \N$ be as in Theorem \ref{t:separation2}. We will construct Borel sets $B_n \subseteq X$ so that every $B_n$ is $f(n)$-separated in $G_n$, and there is a comeager invariant Borel set $C$ such that $C\sse\{x:\exists^\infty n \ (x \in B_n)\}$. 
		
		Fix a bijection $g:\N\to\N^2$ with coordinate functions $g_0,g_1$ and enumerate a basis of $X$ as $(U_n)_{n\in\N}.$ Since \[G_n^{f(n)}=\{(x,y):x\neq y \wedge \dist_{G_n}(x,y)\le f(n) \}\]  has bounded degrees there exists a finite Borel vertex coloring $c_n:X\to m$ of $G_n^{f(n)}$ by \cite{KST}. For some $i\in m$ the preimage $c_n^{-1}(i)$ is non-meager in $U_{g_0(n)}.$ Set $B_n=c_n^{-1}(i).$ This way $B_n$ is $f(n)$-separated in $G_n$ and for every $k\in\N$ we have that
		\[C_k=\bigcup_{n\in g_1^{-1}(k)}B_n\]
		is comeager in $X,$ and 
		\[C'=\bigcap_{k\in\N}C_k\sse \{x:\exists^\infty n \ (x \in B_n)\}.\]
		As the saturation of $X\setminus C'$ is meager, $C =C'\setminus\big[X\setminus C'\big]_E$ is a comeager invariant Borel set such that $\rstr{E}{C}$ is hyperfinite by Theorem \ref{t:separation2}.
	\end{proof}
    
	\section{Open problems}
	\label{s:problems}
	Contrary to the second author's (unpublished) claims made in early 2020, the following question is still open:
	\begin{problem}
		Let $E$ be a CBER on $[\N]^\N$. Is there a Ramsey co-null Borel set $B$ such that $\restriction{E}{B}$ is hyperfinite?
  	\end{problem}
  	It would be already very interesting to exclude constructions used above.
  	  	
	\begin{problem}
		Let $E$ be a CBER on $[\N]^\N$ and $E=\bigcup_n G_n$ where $G_n$ are all bounded degree Borel graphs. Is there necessarily a sequence of Borel sets $B_n$ such that $B_n$ is $f(n)$-separated in $G_n$, where $f$ is a function obeying Theorem \ref{t:separation2} and $\{x:\exists^\infty n \ x \in B_n\}$ is Ramsey co-null?
	\end{problem}

	As mentioned in the introduction, on the Carlson-Simpson space every CBER is actually smooth on some positive set \cite{panagiotopoulos2022every}, while on the Ellentuck space $\mathbb{E}_0$ shows that we cannot expect smoothness. This suggests the following (see also \cite{panagiotopoulos2022every}).

    \begin{problem}
        Characterize those topological Ramsey spaces on which any CBER is smooth on some positive set.
    \end{problem}

	\bibliographystyle{abbrv}
	\bibliography{ref.bib}
	
\end{document}